\documentclass[12pt]{amsart}
\usepackage[all]{xy}
\usepackage{amsmath,amsthm, amssymb, amsfonts, amscd}
\usepackage{verbatim}

\newtheorem{lemma}{Lemma}[section]
\newtheorem{proposition}[lemma]{Proposition}
\newtheorem{theorem}[lemma]{Theorem}

\theoremstyle{definition}

\theoremstyle{remark}

\newtheorem{remark}[lemma]{Remark}

\numberwithin{equation}{section}

\newcommand{\tens}{\otimes}

\newcommand{\iso}{\stackrel{\sim}{\to}}

\newcommand{\ra}{\rightarrow}
\newcommand{\xra}{\xrightarrow}

\newcommand{\Ker}{\operatorname{Ker}}
\newcommand{\Pic}{\operatorname{Pic}}

\newcommand{\Lie}{\operatorname{Lie}}

\newcommand{\ch}{\operatorname{char}}

\newcommand{\Br}{\operatorname{Br}}

\newcommand{\Spec}{\operatorname{Spec}}

\newcommand{\Gal}{\operatorname{Gal}}
\newcommand{\SL}{\operatorname{SL}}

\newcommand{\gSL}{\operatorname{\mathbf{SL}}}

\newcommand{\gGL}{\operatorname{\mathbf{GL}}}

\newcommand{\gSpin}{\operatorname{\mathbf{Spin}}}
\newcommand{\SB}{\operatorname{\mathbf{SB}}}

\newcommand{\End}{\operatorname{End}}
\newcommand{\Hom}{\operatorname{Hom}}
\newcommand{\RHom}{\operatorname{RHom}}

\newcommand{\Ext}{\operatorname{Ext}}
\newcommand{\Coh}{\operatorname{Coh}}
\newcommand{\Qcoh}{\operatorname{Qcoh}}

\newcommand{\Sym}{\operatorname{Sym}}
\newcommand{\Skew}{\operatorname{Skew}}

\newcommand{\pdim}{\operatorname{pdim}}
\newcommand{\rdim}{\operatorname{rdim}}
\newcommand{\gldim}{\operatorname{gldim}}

\renewcommand{\P}{\mathbb{P}}
\newcommand{\Z}{\mathbb{Z}}

\newcommand{\C}{\mathbb{C}}

\newcommand{\gm}{\mathbb{G}_m}

\newcommand{\cA}{\mathcal A}

\newcommand{\cC}{\mathcal C}
\newcommand{\cD}{\mathcal D}
\newcommand{\cE}{\mathcal E}
\newcommand{\cF}{\mathcal F}

\newcommand{\cI}{\mathcal I}
\newcommand{\cJ}{\mathcal J}

\newcommand{\cM}{\mathcal M}
\newcommand{\cO}{\mathcal O}

\newcommand{\cT}{\mathcal T}

\newcommand{\cHom}{\mathcal Hom}
\newcommand{\cRHom}{\mathcal RHom}

\newcommand{\Fol}{\overline{F}}

\title[A Derived Equivalence for some Twisted Projective Homogeneous
  Varieties]{A Derived Equivalence for some Twisted Projective Homogeneous
  Varieties} \author{Mark Blunk} \address{Department of Mathematics,
  Univ. of British Columbia, Vancouver, BC V6T 1Z2, Canada}
\email{mblunk@math.ubc.ca}


\thanks{M. Blunk was supported in part by the National Science
  Foundation, Award No. 0902659}
\begin{document}
\begin{abstract}
In this paper we construct a tilting sheaf for Severi-Brauer Varieties
and Involution Varieties. This sheaf relates the derived category of
each variety to the derived category of modules over a ring whose
semisimple component consists of the Tits algebras of the
corresponding linear algebraic group.
\end{abstract}
\maketitle
 \tableofcontents
\section{Introduction}

The origin of this direction of research is the classic paper
\cite{Bei78} of Beilinson, where it is shown that the derived category
of the projective space $\P_{\C}(V)$ that the sheaves $\cO, \cO(-1),
\dots \cO(-\dim (V)+1)$ form a simple set of generators, (what is
called a \emph{strong exceptional collection}). In \cite{Kap88},
Kapranov generalized these calculations to form strong exceptional
collections for Grassmannians and Quadrics over $\C$. All of these
varieties are \emph{Projective Homogeneous Varieties}. That is, they are of the
form $G/P$, where $G$ is a semisimple linear algebraic group and $P$
is a parabolic subgroup. It is a conjecture of Catanese that every
$G/P$ should possess such a collection, and there have been several
results in this direction (see \cite{Boh06},\cite{Kuz08},
\cite{KuzPol11}, and \cite{FanManPol12}, among others).

In this paper, we will work in another direction. The varieties
considered here are \emph{twisted} projective homogeneous
varieties. These are varieties $X$ defined over a field which, after
extension of scalars to a separable closure, become isomorphic to some
$G/P$. Examples of such varieties are Severi-Brauer Varieties and
Involution Varieties. The group $G$ is of type $A$ for Severi-Brauer
varieties and of type $D$ for Involution Varieties.

Examples of such varieties are Severi-Brauer Varieties and
Involution Varieties. Instead of producing exceptional collections, in
this paper we follow \cite{Bae88} and produce \emph{tilting
  sheaves} (see Section \ref{sec.tilt} for a definition). Our main
result is to produce a locally free sheaf $\cT$ on $X$ which induces a
derived equivalence
\[D^b(\Coh(X)) \ra D^b(\mod-\End (\cT)).\]

This paper is heavily influenced by \cite{Pan94}, where the Quillen
$K$-theory of twisted Projective Homogeneous Varieties is
computed. In \cite{Blu10}, the author produced a locally sheaf $\cF$ on a del
Pezzo surface $S$ of degree 6 which was used to describe the $K$-theory of
$S$. Later, in \cite{BSS11}, it was shown that the same sheaf $\cF$ is
a tilting sheaf for the surface $S$. This paper completes the chain
by showing that the sheaves produced in \cite{Pan94} can be used to
develop derived equivalences for some twisted Projective Homogeneous
Varieties. The arguments in this paper follow a similar line of
reasoning to that found in \cite{BSS11}.

In Section \ref{sec.tilt}, we recall the definition of a tilting
sheaf, our main tool for constructing the desired derived
equivalences. In Section \ref{sec.gldim}, we recall some needed facts
about modules and generation of categories and in \ref{sec.rep}, we
briefly recall some basic properties of semisimple linear algebraic
groups. In particular we state the Borel-Weil-Bott Theorem (Theorem
\ref{thm.BWB}), which we will use repeatedly. In the remainder of the
paper, we construct tilting sheaves for Severi-Brauer Varieties,
Generalized Severi-Brauer Varieties, and Involution Varieties. We
conclude in Section \ref{sec.Ktheory} with an application on computing
the Quillen $K$-theory of these varieties.

\subsection{Notation}
Let us fix some notation. We will let $F$ denote a field, and $\Fol$ a
fixed separable closure of $F$. The group $\Gamma := \Gal(\Fol/F)$ is
the Galois group. By a variety $X$ over $F$ we mean a reduced scheme
of finite type over $\Spec (F)$. The abelian category of quasicoherent
sheaves on $X$ will be denoted $\Qcoh(X)$, and $\Coh (X)$ is the
abelian subcategory of sheaves on $X$. For a ring $R$, $\mod-R$ (resp.
$R-\mod$) is the abelian category of finitely generated right (resp. left)
$R$-modules.

If $\cA$ is an abelian category, then $D(\cA)$ will denote the
corresponding derived category $\cA$ (confer \cite[Chapter
  III]{GelMan03}). This is a triangulated category objects are
complexes with terms in $\cA$, and maps homomorphisms of chain
complexes, modulo homotopy equivalences, and localizing the set of
quasi-isomorphisms. The subcategory of bounded complexes will be
denoted $D^b(\cA)$. If $M \in \cA$, by abuse of notation we will use
the same symbol to denote the complex in $D^b(\cA)$ with $M$
concentrated in degree 0 and every other term equal zero.

\subsection{Acknowledgments}
The author would like to thank Patrick Brosnan, Alexander Merkurjev,
Edward Richmond for many useful discussions.


\section{Tilting Sheaf}
\label{sec.tilt}
We say that a sheaf $\cT$ on a smooth variety $X $ is a \emph{tilting
  sheaf} if the following conditions hold:
\begin{itemize}
\item The sheaf $\cT$ has no self extensions, i.e. $\RHom_{D^b(X)}(\cT[i],\cT)
  = 0$, for $i>0$.
\item The algebra $\End_{\cO_X}(\cT)$ of global endomorphisms of the
sheaf $\cT$ has finite global dimension (see Section \ref{sec.gldim}
for a definition of global dimension of a ring).
\item There is no proper thick subcategory of $D^b(\Coh(X))$
  containing the element $\cT$ (see Section \ref{sec.gen} for a
  definition of thick).
\end{itemize}

Our main tool in this paper is the following theorem.
\begin{theorem}[Theorem 3.12 of \cite{Bae88}]
\label{thm.B}
  Let $X$ be a smooth variety, and let $M \in \Coh(X)$ be a tilting
  sheaf with $S := End_{\cO_X}(M)$.  Then the functors
  \begin{align*}
    F:= \Hom(M, -): \Coh(X) \ra \mod-S\\ 
   G := - \tens M: \mod-S \ra \Coh (X)
  \end{align*}
induce equivalences of categories, 
\begin{align*}
    RF: D^b(\Coh(X)) \ra D^b(\mod-S) \\
   LG: D^b(\mod-S) \ra D^b(\Coh (X)),
\end{align*}
inverse to each other.
  \end{theorem}

\section{Global Dimension}
\label{sec.gldim}
Let $R$ be a ring. The projective dimension of an arbitrary left
$R$-module $T$ is denoted by $\pdim_R T$. The \emph{global (homological) dimension} of $R$ is
supremum of $\pdim_R T$ over all such modules $T$.

\begin{proposition}[\cite{ARS95}, Prop III.2.7]
\label{prop.glDim}
Let $R$ and $S$ be artinian $F$-algebras , and $M$ an
$R$-$S$-bimodule, finitely generated over $F$.  If $S$ is a semisimple
ring, then

\[\gldim \begin{pmatrix} R & M \\ 0  & S \end{pmatrix} = \max \{ \pdim_R M +1, \gldim R \}.\]
\end{proposition}

\subsection{Generation and thick subcategories}
\label{sec.gen}

We recall some properties of generation of triangulated categories
(confer \cite{Nee92}, \cite[Section 4]{BSS11}).

Let $\cD$ be a triangulated category, and let $\cE$ denote a subset of
objects of $\cD$. A triangulated category is equipped with a shift
operator. If $M \in \cD$, the shift of $M$ will be denoted by $M[1]$.

\begin{itemize}
\item A subcategory of $\cC \subset \cD$ is said to be \emph{thick}
  (\'{e}paisse) if it is closed under isomorphisms, shift, taking
  cones of morphisms, and taking direct summands of objects in $\cC$.

\item An object $C \in \cD$ is said to be \emph{compact} if
  $\Hom_{\cD}(C, -)$ commutes with direct sums. Let $\cD^c \subset
  \cD$ denote the subcategory of compact objects in $\cD$.

\item We define $\langle \cE \rangle$ to be the smallest thick full
  subcategory of $\cD$ containing the elements of $\cE$.

\item We define $\cE^{\perp}$ to be the subcategory of $\cD$
  consisting of all objects $M$ such that $\Hom_{\cD}(E[i], M) = 0$,
  for all $i \in \Z$ and all $E \in \cE$.
\end{itemize}

We say that $\cE$ \emph{generates} $\cD$ if $\cE^{\perp} = 0$. If
$\cD^c$ generates $\cD$, then we say $\cD$ is \emph{compactly
  generated}.

If $\cD$ is compactly generated and $\cE \subset \cD^c$. It's
clear that if $\langle \cE \rangle = \cD^c$, then $\cE$ generates
$\cD$.  The following theorem tells us that the converse is also true.

\begin{theorem}[Ravenel and Neeman \cite{Nee92}. Also see Thm. 2.1.2 in \cite{BonVdb03}]
\label{thm.RN}
Let $\cD$ be a compactly generated triangulated category.  Then a set
of objects $E \subset \cD^c$ generates $\cD$ if and only if $\langle E
\rangle = \cD^c$.
\end{theorem}

\section{Groups, Quotients, and Associated Sheaves}
\label{sec.rep}
We briefly summarize some definitions and properties of linear
algebraic groups, which we will need in the paper. Some references for
this section are \cite{Bou02}, \cite[I.5]{Jan03}, \cite{Spr98}, and
\cite[Chapter 24]{KMRT98}.

Let $G$ denote a split, semisimple linear algebraic group over $F$,
and $T \subset G$ a fixed maximal torus. The group $\Lie(T)^*: = \Hom
(T, \gm)$ will denote the \emph{character lattice} of $G$. It is a
free $\Z$-module of finite rank. The set $R \subset \Lie(T)^*$ will denote a
root system corresponding to $T$, $R^+$ the set of positive roots,
and $R^-:=-R^+$ the set of negative roots, and the set $S = \{\alpha_1, \dots,
\alpha_r\} \subset R^+$ is a basis of simple roots.

For $\alpha \in \Lie(T)^*$, let $\Lie(G)^\alpha$ denote the eigenspace
of $\Lie(G)$ corresponding to $\alpha$.  The Lie subalgebra
\[\Lie(T) \bigoplus_{\alpha \in R^+} \Lie(G)^\alpha\]
is $\Lie(B)$ for a fixed Borel subgroup $B$ of $G$ determined by $T$ and
$S$. If $I \subset S$, then the Lie subalgebra
\[\Lie(B) \oplus  \bigoplus_{\alpha \in R^-(I)} \Lie(G)^\alpha,\]
where \[R^-(I):= \{\alpha \in R^- |\ \ \alpha = \Sigma_{i=1}^r
a_i\alpha_i\ \mathrm{with}\ a_i \leq 0,\ \ a_i = 0\ \forall\ \alpha_i \in
I \},\] is $\Lie(P_I)$, for a unique parabolic subgroup $P_I \supset B$
of $G$. Every such parabolic subgroup (intermediate between $B$ and
$G$) arises in this fashion. For example, $P_\emptyset = G$ and $P_S = B$.

Let $\{\lambda_1,\dots, \lambda_r\}$ be the set of \emph{fundamental
  weights} determined by the simple roots $\alpha_i$, and $\Lambda = :\Z[\lambda_i]$ be the \emph{weight
  lattice}. If $P_I$ is a parabolic subgroup of $G$, then we say a
weight $\lambda$ is \emph{dominant for $P_I$} if \[\lambda = \Sigma_{i
  \in I} n_i \lambda_i + \Sigma_{j \notin I} n_j
\lambda_j,\ \ \mathrm{where} \ n_j \geq 0.\] Also, let $\rho =
\Sigma_i \lambda_i$.

The Weyl group $W$ is the group generated by the simple reflections
$s_\alpha$ corresponding to $\alpha \in S$. For $w \in W$, the
\emph{length} of $w$ is the least number of factors in a decomposition
of $w$ as a product of simple reflections. There is an an action of
$W$ on the weight lattice $\Lambda$. We will need another action,
called the \emph{dot} or \emph{affine} action on $\Lambda$:
\[w.\lambda :=w(\lambda +\rho) -\rho.\] We say that a weight $\lambda$
is \emph{singular} if there is some non-trivial $w \in W$ such that
$w.\lambda = \lambda$.

If $P \subset G$ is a parabolic subgroup, there exists a
decomposition, called the Levi decomposition, of $P$ into a semisimple
or Levi factor $L_P$ and a unipotent subgroup $R_u(P)$. If $\phi: P \ra
\gGL(V)$ is an irreducible finite-dimensional representation, $R_u(P)$
acts trivially, and so $\phi$ descends to a representation of the Levi
factor $L_{P}$. Each such representation posses a unique highest weight
$\lambda \in \Lambda$ which determines the representation. Moreover
this weight is dominant for $P$. For a weight $\lambda$, we will
denote the corresponding representation vector space by $V(\lambda)$.

Finally, if $\phi: P \ra \gGL(V)$ is as in previous paragraph, with corresponding
weight $\lambda$, then we define a locally free sheaf of rank
$\dim(V)$ on the projective homogeneous variety $G/P$ as follows:
\[G \times_r V := G \times V / \{(g, v) \sim (gp^{-1}, \phi(p)(v))\ |
\ p \in P,\ g\in G,\ v \in V \}.\] The projection $G \times V \ra
G$ induces a map $G \times_r V \ra G/P$, defining a vector bundle over
$G/P$. The corresponding locally free sheaf on $G/P$ will be denoted
$\cO_{G/P}(\lambda)$.  A section of this sheaf can be thought of as a
function $F:G \ra V$ satisfying $F(gp) = \phi(p^{-1})F(g)$.

Let us recall the celebrated Borel-Weil-Bott theorem, which relates
representations on $P$ to the cohomology of the induced sheaf on the
variety $G/P$. 
\begin{theorem}[Theorem 5.0.1 of \cite{BasEas89}]    
\label{thm.BWB}
Let $G$ be a simply connected split semisimple algebraic group,
$P \subset G$ be a parabolic subgroup, and $\lambda \in \Lambda$ is
dominant with respect to $P$. Consider the corresponding sheaf
$\cO_{G/P}(\lambda)$ on $G/P$. Then:
\begin{itemize}
\item[a] If $\lambda$ is singular, \[H^r(G/P, \cO_{G/P}(\lambda)) =
  0,\ \ \forall r.\]
  
\item[b] If $\lambda$ is not singular, then there exists a unique $w
  \in W$ such that $w.\lambda$ is dominant. Moreover,
\begin{align*} H^i(G/P, \cO_{G/P}(\lambda)) &= 0, \  i \neq l(w)\\
                      H^{l(w)}(G/P, \cO_{G/P}(\lambda)) &= V(w.\lambda). 
\end{align*}
\end{itemize}
\end{theorem}

\section{Severi-Brauer Varieties}
\label{sec.SB}
We recall some well known facts about central simple algebras (confer
\cite{Art82} and \cite[Section I.1.B]{KMRT98}). Let $A$ be a \emph{central
simple algebra} over $F$ of degree $n$. The algebra $A$ has dimension
over $F$ equal to $n^2$, has no non-trivial two sided ideals, and center
equal to $F$. Equivalently, $A\tens_F \Fol$ is isomorphic to
$\End_{\Fol}(V')$, for some $\Fol$-vector space $V'$ of dimension
$n$. The algebraic group $\SL_1(A)$ is of type $A_n$. Every finitely
generated right $A$-module has dimension, as an $F$-vector space,
divisible by $n$, and we define the \emph{reduced dimension} of $M$ by
\[\rdim_A (M) := \frac{\dim_F(M)}{n}.\] Finally, we say that the algebra $A$ is
\emph{split} if $A = \End_F(V)$, where $V$ is a finite dimensional
$F$-vector space.

Let $X := \SB(A)$ be the \emph{Severi-Brauer} variety of the algebra
$A$. This is an irreducible, smooth, projective variety of dimension
$n-1$, whose points consist of right ideals of $A$ which have reduced
dimension 1. If $E/F$ is a field extension, $\SB(A)_E = \SB(A\tens_F
E)$. When $A = \End_F (V)$ for some $F$-vector space $V$ of dimension
$n$, every right ideal of $A$ of reduced dimension 1 is determined by
a unique 1-dimensional subspace of $V$. Thus $\SB(\End_F(V)) =
\P(V)$. In particular, Since $P(V) = G/P$ for $G = \gSL (V)$ and
$P=P_{\alpha_1}$ the stabilizer of a line in $V$, we see that $\SB(A)$
is a twisted form of a projective homogeneous variety.

We define the `tautological' sheaf $\cI$ on $X$ (confer \cite{Art82},
\cite[section 10.2]{Pan94}), a subsheaf of the consist sheaf $A$. The
fiber over a closed point $x \in X(\Fol)$ consists of the elements $a
\in A_{\Fol}$ such that $a \in x \subset A_{\Fol}$ (here $x$ is a
right ideal of $A_{\Fol}$ of reduced dimension 1). This defines a
locally free sheaf of rank $n$ on $X$.  This sheaf $\cI$ is locally
free, and has an induced right $A$ action, and the algebra $\End(\cI)$
is isomorphic to the algebra $A$. Finally, If $A = \End_F(V)$ is
split, then $\cI = \cO_{\P(V)}(-\lambda_1) \tens_F V^*$.

Finally, let
\[\cT = \cO_X \oplus \cI^1 \oplus \cI^2 \oplus \dots \oplus \cI^{\tens (n-1)}.\]  
We will show that $\cT$ is a tilting sheaf for $X$.

\begin{theorem}
The sheaf $\cT$ has no self extensions,
i.e., \[\RHom_{D^b(\Coh(X))}(\cT[i],\cT) = 0,\] for $i>0$.
\label{thm.SBext}
\end{theorem}

\begin{proof}
It suffices to extend scalars to $\Fol$, so we may assume that $F$ is
separably closed. In this case, $X = \P(V)$, and the sheaf $\cT$
decomposes into a sum of invertible sheaves of the form
$\cO_{\P(V)}(-j\lambda_1)$, where $j = 0,1,\dots, n-1$.

So it suffices to show that 
\[\RHom_{D^b(\Coh(\P(V)))} (\cO_{\P(V)}(-j_1\lambda_1)[i],
\cO_{\P(V)}(-j_2\lambda_1)) = 0,\] for $i >0$ and $j_1, j_2 = 0, \dots n-1$.
But
\begin{align*}
\RHom_{D^b(\Coh(\P(V)))} (\cO_{\P(V)}(-j_1\lambda_1)[i],
\cO_{\P(V)}(-j_2\lambda_1)) = \\
\Ext^i_{\P(V)}(\cO_{\P(V)}(-j_1\lambda_1),\cO_{\P(V)}(-j_2\lambda_1)) =\\
 H^i(\P(V),\cO_{\P(V)}((j_1-j_2)\lambda_1)) = 0,
\end{align*}
for $-n < j_1 - j_2 \leq n$ and $i > 0$. This follows from Theorem \ref{thm.BWB},
since the corresponding weight in each case is either dominant or singular.
\end{proof}

\begin{remark}
\label{rem.SBut}
Also, note that \[\Hom_{\P(V)}(\cO_{\P(V)}, \cO_{\P(V)}(-i\lambda_1)) =
H^0(\P(V), \cO_{\P(V)}(-i\lambda_1)) = 0,\] since the weight
$-i\lambda_1$ is singular when $1\leq i \leq n-1$. We will need this in
the proof of Theorem \ref{thm.SBgldim}.
\end{remark}

\begin{theorem}
The algebra $\End(\cT)$ has finite global dimension.
\label{thm.SBgldim}
\end{theorem}

\begin{proof}
We prove by induction on $n$. The base case $n=1$ is trivial, as the
field $F = \End (\cO_{pt})$ has global dimension 0.

Because $\cT = \oplus_{i=0}^{n-1} \cI^{\tens i}$, we have the
following matrix presentation for $\End(\cT)$ (by Remark
\ref{rem.SBut}, this matrix is lower triangular).

\[\End(\cT) = \begin{pmatrix} 
F & \Hom(\cI,\cO_X) & \Hom(\cI^{\tens
  2}, \cO_X) &   \cdots &  * \\
0 & A & * & \cdots & * \\
0 & 0 & A^{\tens 2} & \cdots & * \\
\vdots & \vdots & \vdots & \ddots & \vdots \\
0 & 0 & 0 & \cdots & A^{\tens n-1})
\end{pmatrix}.\]

So we can write 
\begin{align*}
\End_X(\cT) & =\begin{pmatrix} R & B \\ 0 &
  (A^{\tens n-1})
\end{pmatrix},
\ \ \mathrm{where} \\
R &= \begin{pmatrix} 
F & * & * & \cdots & * \\
0 & A & * & \cdots & * \\
0 & 0 & A^{\tens 2} & \cdots & * \\
\vdots & \vdots & \vdots & \vdots & \vdots \\
0 & 0 & 0 & \cdots & A^{\tens n-2}
\end{pmatrix},\\
B & = \begin{pmatrix} 
\Hom (\cO_X, \cI^{\tens n-1}) \\
\vdots \\
\Hom (\cI^{\tens n-2}, \cI^{\tens n-1}) 
\end{pmatrix}.
\end{align*}

By induction, the ring $R$ has finite global dimension, and in
particular, $\pdim_R(B)$ is finite. Since $A^\tens{(n-1)}$ is
semisimple, we have by Proposition \ref{prop.glDim} that $\gldim
(\End_X(\cT)) = \max \{\pdim_R(B)+1, \gldim (R)\}$, and we conclude
that $\End_X(\cT)$ has finite global dimension.
\end{proof}

\begin{lemma}
\label{lem.SBsplitgen} Assume that the algebra $A$ is split. The 
sheaf $\cT$ generates $D(\Qcoh(\SB(A))$, and $\langle
\cT \rangle = D^b(\Coh (\SB(A)))$.
\end{lemma}

\begin{proof}
Since $A = \End_F(V)$ is split, $\SB(A) = \P(V)$, and $\cI =
\cO_{\P(V)}(-\lambda_1) \tens V^*$. The sheaf $\cO_{\P(V)}(-i\lambda_1)$
is a summand of $\cT$, for $i = 0, \cdots = n-1$. By \cite{Bei78} (or
\cite[Section 3]{Kap88}), we know that these invertible sheaves form a
strong exceptional collection on $\P(V)$. In particular,
\[\langle \big \{\cO_{\P(V)}(-i \lambda_1)\ \ 0\leq i\leq n-1 \big\}\rangle = D^b (\Coh (\P(V))).\]
This forces $\langle \cT \rangle = D^b(\Coh(\P(V)))$, and hence
the sheaf $\cT$ generates $D(\Qcoh(\P(V)))$.
\end{proof}
\begin{proposition}
The sheaf $\cT$ generates $D (\Qcoh(\SB(A)))$.
\label{prop.SBgen}
\end{proposition}

\begin{proof}
Let $\cM \in D(\Qcoh (\SB(A)))$ and assume that
\[\RHom_{D(\Qcoh(SB(A)))}(\cT, \cM) = 0.\] Since $\cT$ is locally free,
$\cHom(\cT, -)$ and $T^* \tens -$ are exact functors on
$\Qcoh(X)$. (Similarly for $\cHom(\overline{\cT}, -)$ and
$\overline{T}^* \tens -$ on $\Qcoh(\overline{X})$.) Thus, for example,
$\cRHom_X(\cT, \cM)$ can be computed on $D(\Qcoh(X))$ by applying
$\cHom(\cT, -)$ to each individual term in $\cM$.

Consider the following cartesian square:
$$\xymatrix{
\SB(A)_{\Fol} \ar[rr]_{v} \ar[d]_{q} && \SB (A) \ar[d]_{q}
\\
\Spec (\Fol) \ar[rr]_{u} && \Spec (F)
}.$$
Since $u$ (and thus $v$) is flat, it follows that the natural map
\[u^*Rp_* \ra Rq_*v^*\] is an isomorphism of functors (see \cite[(3.18)]{Huy06}).

\begin{align*}
0 & = u^*(\RHom_X(\cT,\cM)) \\
  & = u^*Rp_* \cRHom_X(\cT,\cM)  \quad \hbox{By \cite{Huy06}, page 85} \\
 & = Rq_*v* \cRHom_X(\cT,\cM)  \\
 & = Rq_*v* (\cT^{\vee} \tens^L_X \cM)  \\
 & = Rq_* (\overline{\cT}^{\vee} \tens^L_{\overline{X}} v^*\cM)  \\
& = Rq_* \cRHom_{\overline{X}}(v^*\cT,v^*\cM)  \\
& = \RHom_{\overline{X}}(v^*\cT,v^*\cM).
\end{align*}

The algebra $A_{\Fol}$ splits, and thus $v^*\cT = \overline{\cT}$
generates $D(\Qcoh(\overline{X}))$, by Lemma
\ref{lem.SBsplitgen}. This implies that $v^*(\cM) = 0$. Since $v$ is
flat, this forces $\cM = 0$. Hence, $\cT$ generates $D(\Qcoh(X))$.
\end{proof}

We conclude by collecting our results to prove the main theorem of
this section.
\begin{theorem}
\label{thm.SBEquiv}
The map
\[R\Hom(\cT, -): D^b(\Coh (X))\ra D^b(\End(\cT)-\mod)\]
is an equivalence of categories.
\end{theorem}

\begin{proof} 
The sheaf $\cT$ generates $D(\Qcoh(X))$ by Proposition \ref{prop.SBgen}, and
since $D(\Qcoh(X))$ is compactly generated, it follows that
$\langle\cT\rangle = D^b(\Coh(X))$ by Theorem \ref{thm.RN}. The sheaf $\cT$
has no-self extensions by Theorem \ref{thm.SBext}.  The algebra $A
= \End(\cT)$ has finite global dimension, by Theorem
\ref{thm.SBgldim}. So $\cT$ is a tilting sheaf, and the theorem
follows from \cite[Theorem 3.1.2]{Bae88}.
\end{proof}

\begin{remark}
In \cite{Ber09}, the author produces a semi-orthogonal decomposition
for the derived category of a Severi-Brauer scheme over the derived
category of the base, by producing a collection of twisted sheaves
which satisfy the necessary properties. In the case where the base is
$\Spec(F)$, this is equivalent to result here, since the twisting data
comes from a Brauer class, i.e. the central simple algebra $A$.
\end{remark}
\section{Generalized Severi-Brauer Varieties}
\label{sec.GSB}

As in the previous section, let $A$ be a central simple algebra of
degree $n$. Let $X = SB(r,A)$ be the \textit{Generalized Severi-Brauer
  Variety}, for some $0<r<n$. The points of $X$ are the right ideals
$I \subset A$ of reduced dimension $r$. Obviously, $SB(1,A)$ is just
the usual Severi-Brauer variety.

Let $\cI$ be the tautological sheaf of $\SB(r,A)$, defined in an
analogous fashion to the sheaf in Section \ref{sec.SB}. This is a
locally free sheaf of rank $rn$. When $A = \End(V)$ is split,
$\SB(r,A) = \gSL(V)/ P_{\alpha_r}$, and $\cI = \cO_{G/P}(-\lambda_r)
\tens V^*$.  Let
\[\cT = \bigoplus_a \Sigma^a (\cI).\]
Here $a = (a_1, \dots, a_r) \in \mathbb{N}^r$, subject to the
condition $n\geq a_1 \geq a_2 \geq \dots \geq a_r \geq 0$ (that is,
Young diagrams with at most $n-r$ rows and at most $r$ columns) and $\Sigma$ is the
Schur Functor corresponding to $a$. Finally, let $d(a) =
a_1+\dots+a_n$.

\begin{theorem}
The sheaf $\cT$ has no self extensions.
\label{thm.GSBext}
\end{theorem}

\begin{proof}
The proof is the same as in the proof of Theorem \ref{thm.SBext}. The
only difference is that the weight that appears in the argument is
$\lambda_r$, instead of $\lambda_1$.
\end{proof}

\begin{theorem}
The ring $\End(\cT)$ has finite global dimension.
\label{thm.GSBgldim}
\end{theorem}

\begin{proof}
The ring $\End(\cT)$ is upper triangular, with diagonal entries of the
form $A^{d(a)}$, which are all semisimple. The proof follows
as in the proof of \ref{thm.SBgldim}.
\end{proof}

\begin{proposition}
The sheaf $\cT$ generates $D(\Qcoh(X))$, and hence $\langle \cT
\rangle = D^b(\Coh(X))$.
\label{prop.GSBgen}
\end{proposition}

\begin{proof}
It suffices to show that $\cT$ generates $D(\Qcoh(X))$ in the case
where $A$ is split. In that case, $\cT$ contains terms of the form
$\Sigma^a(\cO_{G/P_{\alpha_r}}(-\lambda_r))$. In \cite[Theorem
  3.4]{Kap88}, it is shown that these sheaves form a strong
exceptional collection for $G/P_{\alpha_r}$. In particular, 
\[\langle \Big\{\Sigma^a(\cO_{G/P_{\alpha_r}}(-\lambda_r)), a\Big\}\rangle =
D^b(\Coh(G/P_{\alpha_r})).\] It follows that $\cT$ generates
$D(\Qcoh(X))$, and by Theorem \ref{thm.RN}, $\langle \cT \rangle = D^b(\Coh(X))$.
\end{proof}

\begin{theorem}
\label{thm.GSBEquiv}
The sheaf $\cT$ is a tilting bundle, and thus 
\[\cRHom(\cT, -): D^b(\Coh (X))\ra D^b(\mod-\End(\cT))\]
is an equivalence of derived categories.
\end{theorem}

\begin{proof}
The sheaf $\cT$ has no self-extensions by Theorem \ref{thm.GSBext},
$\langle \cT \rangle = D^(\Coh(X))$ by Proposition \ref{prop.GSBgen},
and the algebra $\End(\cT)$ has finite global dimension by PROPOSITION
\ref{prop.GSBgen}. Thus $\cT$ is a tilting sheaf, and the result
follows by \ref{thm.B}.
\end{proof}

\section{Involution Varieties}
\label{sec.Inv}
In this section we assume that $\ch (F) \neq 2$.

Let $(A, \sigma)$ be a central simple algebra of degree $2n$, equipped
with an orthogonal involution $\sigma$. Recall (\cite[Proposition
  2.6]{KMRT98} that $\sigma$ is \emph{orthogonal} if
\begin{align*}
\dim_F \Sym(A,\sigma) &= \frac{2n(2n+1)}{2}, \\
\dim_F \Skew(A, \sigma) & = \frac{2n(2n-1)}{2}.
\end{align*}
Let $X := I(A,\sigma)$ be the
\emph{involution variety} of $A$ (confer \cite{Mer95},
\cite{Tao94}). This is a codimension 1 subvariety of $\SB(A)$, whose
points consists of the right ideals $I \in \SB(A)$ which are
\emph{orthogonal}, i.e. $\sigma(I) \cdot I = 0$.  We will let $\cI \in
\Coh(I(A, \sigma))$ to denote the pullback of the tautological bundle
of $\SB(A)$ to $I(A, \sigma)$.

Recall from that introduction that $\Gamma = \Gal(\Fol/F)$. If $\cF$
is a sheaf for the \'{e}tale topology, we have the Hochschild-Serre
Spectral Sequence: \cite[Section 2]{Art82}
\[H^p(\Gamma, H^q(\overline{Y}, \cF)) \ra H^{p+q}(Y, \cF).\]
If $Y = \SB(A)$, then $\overline{Y} = \P(V')$ for some vector space
$V'$ over $\Fol$. By looking at the edge terms of the spectral sequence, we have the following exact sequence:
\[0 \ra \Pic (\SB(A)) \ra \Pic(\P(V'))^\Gamma  \xra{f} \Br(F). \]
Here $\Pic(\P(V_{\Fol}))^\Gamma = \Z^\Gamma = \Z$, generated by the
invertible sheaf $\cO_{\P(V_{\Fol})}(-\lambda_1)$. The map $f$ sends
$\cO_{\P(V_{\Fol})}(\lambda_1)$ to the class of $A$ in the Brauer
group $\Br(F)$.  Since $A$ has an orthogonal involution $\sigma$,
the exponent of $A$ divides 2. It follows that
$\cO_{\P(V_{\Fol})}(-2\lambda_1) \in \Ker(f)$ descends to an
invertible sheaf on $\SB(A)$. Restricting to $X$, we get an invertible
sheaf which we will label $\cO_X(-2\lambda_1)$.

Let $C(A,\sigma)$ denote the \emph{Clifford Algebra} associated to the
pair $(A, \sigma)$ (confer \cite[II.8.7]{KMRT98}).  Since
$C(A,\sigma)$ is defined as a quotient of the tensor algebra of $A$,
there exists a canonical $F$-linear map $c:A \ra C(A,\sigma)$
(confer \cite[(II.8.13)]{KMRT98}).

We define a subsheaf $\cJ$ of the constant sheaf $C(A,\sigma)$ on $I(A,
\sigma)$. If $I \in I(A,\sigma)$ is an isotropic right ideal of
reduced dimension 1, then the fiber over $I$ is the right ideal of
$C(A,\sigma)$ generated by $c(I)$. The endomorphism ring $\End_X(\cJ)$ is isomorphic to $C(A, \sigma)$

\begin{remark}
If $A = \End (V)$ for some finite dimensional vector space $V$ over
$F$, then $\sigma = \sigma_q$ is the adjoint involution with respect
to some non-singular quadratic form $q \in S^2(V^*)$. In this case,
The isomorphism $\SB(\End(V)) = \P(V)$ identifies $I(\End(V),
\sigma_q)$ with the quadric $Z(q)$. Also, if $q$ is maximally
isotropic (i.e. there exists an isotropic subspace of dimension $n$ in
$V$), then $I(\End(V), \sigma_q) = G/P$, where $G = \gSpin(q)$ is a
split group of type $D_n$, and $P = P_{\alpha_1}$. So $I(A, \sigma)$
is a twisted projective homogeneous variety. We say that $(A, \sigma)$
is split if $A = \End(V)$ and $\sigma = \sigma_q$, where $q$ is
maximally isotropic.  Finally, when $X = G/P$, $\cJ =
\cO_{G/P}(-\lambda_{n-1}) \tens_F W_+^* \oplus \cO_{G/P}(-\lambda_{n})
\tens_F W_-^*,$ where $W_+ = V(-\lambda_{n-1})$ and $W_- =
V(-\lambda_{n})$, the two half-spin representation spaces associated to
the weights $-\lambda_{n-1}$ and $-\lambda_n$ (confer \cite[page 574]{Pan94}).
\end{remark}

Let
\[\cT := \bigoplus_{i=0}^{n-2} \Big (\cO_X(-2\lambda_1)^{\tens i} \oplus \cI
\tens\cO_X(-2\lambda_1)^{\tens i} \Big ) \bigoplus \cJ \tens \cO_X
(-2\lambda_1)^{\tens (n-1)}.\] We will show that $\cT$ is a tilting bundle for
$I(A, \sigma)$.


\begin{proposition}
\label{prop.invGen}
The sheaf $\cT$ generates $D(\Qcoh(X))$ 
\end{proposition}
\begin{proof}
This is similar to the proof of Proposition \ref{prop.SBgen}. We first
check that the proposition is true in the split case, where the
argument is similar to Lemma \ref{lem.SBsplitgen}. So we may assume
that $A = \End_F(V)$ and $\sigma = \sigma_q$ for a maximally isotropic
quadratic form $q\in S^2(V^*)$, $I(A, \sigma) = Z(q) =
\gSpin(q)/P_{\alpha_1}$. The sheaf $\cT$ decomposes into a sum
with terms $\cO_{G/P}(-i\lambda)$, $\cO_{G/P}(-2(n-1)\lambda_1
-\lambda_{n-1})$ and $\cO_{G/P}(-2(n-1)\lambda_1 - \lambda_n)$, where
$0 \leq i \leq 2n-3$.  It is shown in \cite[section 4]{Kap88} that
these sheaves form a strong exceptional collection. In particular,
\begin{align*}\Big\langle 
  \cO_{G/P}(-i\lambda_1), \cO_{G/P}(-(2n-2)\lambda_1 -\lambda_{n-1}), 
\cO_{G/P}(-(2n-2)\lambda_1 - \lambda_n),
 \Big\rangle \\= D^b(\Coh (Z(q))).
\end{align*}
Thus we conclude that $\cT$ generates $D(\Qcoh(X))$ in the split
case. The arbitrary case follows by extending scalars to $\Fol$, and
reasoning as in Proposition \ref{prop.SBgen}.
\end{proof}

\begin{theorem}
The sheaf $\cT$ has no self-extensions.
\label{thm.InvExt}
\end{theorem}

\begin{proof}
We extend scalars to $\overline{F}$, arguing as in Theorem
\ref{thm.SBext}. The sheaf $\overline{\cT}$ decomposes into a sum of
sheaves with terms $\cO_{G/P}(-(2n-2)\lambda_1 -\lambda_{n-1})$,
$\cO_{G/P}(-(2n-2)\lambda_1 - \lambda_n)$, and $\cO_X(j\lambda_1)$,
where $-(2n-2) < j \leq 0$. So it suffices to check that
\begin{align*}
H^i(\overline{X}, \cO_X(\pm j\lambda_1)) &= 0, \\
H^i(\overline{X}, \cJ_1^* \tens \cO_X(j\lambda_1)) &= 0, \\
H^i(\overline{X}, \cJ_1 \tens \cO_X(-j\lambda_1)) &= 0, \\
H^i(\overline{X}, \cJ_2^* \tens \cO_X(j\lambda_1)) &= 0, \\
H^i(\overline{X}, \cJ_2 \tens \cO_X(-j\lambda_1)) &= 0, \\
 H^i(\overline{X}, \cJ_1 \tens \cJ_2^{*}) &= 0, \\
 H^i(\overline{X}, \cJ_2 \tens \cJ_1^{*}) &= 0,
\end{align*}
for $i > 0$ and $-(2n-2) < j \leq 0$.

The highest weights of the corresponding $P$-representations for these
sheaves are, respectively, 
\begin{align*}
\pm j\lambda_1, \\
\pm (-\lambda_n +j\lambda_1), \\
\pm (-\lambda_{n-1} +j\lambda_1), \\
\pm ( \lambda_{n-1} -\lambda_n). \\
\end{align*}  All of these weights are either singular or
dominant. Thus by Theorem \ref{thm.BWB}, all of the non-zero
cohomology above vanishes.
\end{proof}

\begin{remark}
As in the Severi-Brauer case, all of the weights $j\lambda_i$ are singular
for $j < 0$, so $H^0(\overline{X}, \cO_{\overline{X}}(j\lambda_1)) =
0$. It follows that we have the following upper triangular presentation of
the global endomorphism rings $\End (\cT):$
\[\End(\cT) = \begin{pmatrix} 
F      & *      & *      & *      & \cdots & * \\
0      & A      & *      & *      & \cdots & * \\
0      & 0      & F      & *      & \cdots & *  \\
\vdots & \vdots & \vdots & \ddots & \vdots & \vdots \\
0      & 0      & 0      &\cdots  & A      &  0 \\
0      &  0     & 0      & \cdots & 0      &  C(A, \sigma)
\end{pmatrix}.\]
As in the Severi-Brauer case, we see that $\End(\cT)$ is upper
triangular, with the Tits Algebras appearing along the diagonal terms.
\end{remark}

\begin{theorem}
\label{thm.invGldim}
The ring $\End(\cT)$ has finite global dimension.
\end{theorem}
\begin{proof}
The proof follows the same line of reasoning as in Theorem
\ref{thm.SBgldim}, since all diagonal terms are semisimple algebras.
\end{proof}

\begin{theorem}
\label{thm.InvEquiv}
The sheaf $\cT$ induces a natural equivalence
\[ \RHom(\cT, -): D^b(\Coh(X)) \ra D^b(\mod-\End(\cT)).\] 
\end{theorem}
\begin{proof}
We need to verify that $\cT$ is a tilting sheaf. It has no self
extensions by Theorem \ref{thm.InvExt}. By Proposition
\ref{prop.invGen} and Theorem \ref{thm.RN}, we see that $\langle
\cT \rangle = D^b(\Coh (X))$. The ring $\End(\cT)$ has finite global
dimension by Theorem \ref{thm.invGldim}, and so the statement follows
from Theorem \ref{thm.B}.

\end{proof}

\section{K-theory}
\label{sec.Ktheory}
In each case discussed above, the ring $\End(\cT)$ has an upper
triangular presentation, with the Tits Algebras of the corresponding
simply connected linear algebraic group appearing along the
diagonal. Thus in each case, $\End(\cT)$ a nilpotent ideal $I$,
consisting of the strictly upper triangular terms. By applying the
$K$-theory functor (confer \cite[Theorem 1.98]{ThoTro90}) to the
natural equivalences found in Theorems \ref{thm.SBEquiv},
\ref{thm.GSBEquiv}, and \ref{thm.InvEquiv}, we can express the Quillen
$K$-theory of each variety as sum of the $K$-theory of the algebras
appearing along the diagonal. The $K$-theory is not affected if we
replace $\End(\cT)$ by $\End(\cT)/I$, and we recovers results found in
\cite[10.2, 10.3]{Pan94}.

\begin{theorem}
The tilting bundles induce the following isomorphisms:
\begin{align*}K_*(\SB(A)) \iso& \bigoplus_{i=0}^{n-1} K_* (A^{\tens i})\\
K_*(\SB(r, A)) \iso &\bigoplus_{a} K_* (A^{\tens d(a)})\\
K_*(I(A,\sigma)) \iso& \Big  (\bigoplus_{i=0}^{n-2}K_*(F) \oplus
K_*(A)\Big )\bigoplus K_*(C_0(A, \sigma)).
\end{align*}
\end{theorem}

\bibliography{Bib}    
\bibliographystyle {amsplain}
\end{document}